\documentclass[12pt]{article}

\begin{document}

\begin{titlepage}

\vspace*{-2cm}

\vspace{.5cm}

\begin{centering}

\huge{Acid zeta function and ajoint acid zeta function}

\vspace{.5cm}

\large  {Jining Gao }\\

\vspace{.5cm}

 Department of Mathematics, Shanghai Jiaotong University, Shanghai, P. R. China

Max-Planck Institute for Mathematics, Bonn, Germany.

\vspace{.5cm}

\begin{abstract}

In this paper we set up  theory of acid zeta function and ajoint acid zeta function ,
based on the theory, we point out a reason to doubt the truth of Riemann hypothesis and also
as a consequence, we give out some new RH equivalences .

\end{abstract}

\end{centering}

\end{titlepage}

\pagebreak

\def\lh{\hbox to 15pt{\vbox{\vskip 6pt\hrule width 6.5pt height 1pt}
  \kern -4.0pt\vrule height 8pt width 1pt\hfil}}
\def\blob{\mbox{$\;\Box$}}
\def\qed{\hbox{${\vcenter{\vbox{\hrule height 0.4pt\hbox{\vrule width
0.4pt height 6pt \kern5pt\vrule width 0.4pt}\hrule height
0.4pt}}}$}}

\newtheorem{theorem}{Theorem}
\newtheorem{lemma}[theorem]{Lemma}
\newtheorem{definition}[theorem]{Definition}
\newtheorem{corollary}[theorem]{Corollary}
\newtheorem{proposition}[theorem]{Proposition}
\newcommand{\proof}{\bf Proof.\rm}

\section{Introduction}

The motivation of constructing acid zeta function is to study distribution of Riemann zeros.
Let's take a look at how it comes up . First of all, we can consider  Melling transformation of 
$N(T)$, the distribution of non-trivial Riemann zeros and it's a meromorphic function so called adjoint acid 
zeta function on some half plane. Of course, we are eager to get full information of this  function,but
unfortunately, we can not know it further without RH. That's why we have to consider another function
so called acid zeta function which can be explicitly figured out . Under the RH ,  Acid zeta function is
equal to ajoint acid zeta function. With some deep discussion on acid zeta function, we will see 
 maginitude of acid zeta function increases along vertical direction with at most polynomial increasing,
when RH is true but explicit formula of acid zeta function contains an  isolated term which has expotient increaseness along
 vertical direction and we have reseason to doubt the truth of RH.This paper has three sections. 
In the first section, we will introduce acid zeta function and study it's properties in detail including analytic continuation ,explicit formula. 
In the second section we will introduce ajoint acid zeta function and give out a formula so called "half explicit" formula   .
In the last section we  set up a remarkable formula which reflects relationship between the acid zeta function and the adjoint acid
zeta function, as a consequence, we prove a "numberical" RH equivalence and point out infinite number of "numerical"
equivalences can be obtained in same way.  Our main tools in this  paper  are one variable complex analysis and some basic 
Riemann zeta function theory .  Most theorems of this paper are obtained through many complicated residues 
compution, to avoid over complicated version, we add an appendix to help reader to study more easily.

\section{Acknowledgements}
This paper is finished during the author visiting at Max-Palanck institute for mathematics(MPIM) in Bonn and is
supported by research grant from MPIM, the author feel happy to have a chance to express many thanks to MPIM.

\section{Acid zeta function}
In this section we will define and study properties of acid zeta function.
{\bf Definition.} Let $\zeta$ be the Riemann zeta function,$\rho_i$ be all nontrival Riemann zeros which satisfy 
$Im \rho_{i}> 0$.Suppose $Re s>1$, we define a function of $s$ as follows:

\begin{eqnarray}
\zeta_{a}(s)=\sum_{m}\frac{1}{(\frac{\rho_{m}-\frac{1}{2}}{i})^{s}}\nonumber
\end{eqnarray}

so called acid zeta function.
\bigskip

Obviously, $\zeta_{a}(s)$ is analytic in the region $Re s>1$,in order to extend domain of this function to the region $Re s<1$ as a meromorphic function , we need following theorem.

\begin{theorem}
When $1<Res<2$, we have that
\begin{eqnarray}
\zeta_{a}(s)= \frac{sin\frac{\pi}{2}s}{(s-1)\pi}\int_{\frac{1}{2}}^{+\infty}\frac{\psi(t+\frac{1}{2})}{t^{s-1}}dt\nonumber
\end{eqnarray}
Where $\psi(t)=\frac{d}{dt}(\frac{\xi^{'}}{\xi}(t))$

\end{theorem}
\begin{proof}
Let $\Gamma_{r}^T$ be the positively oriented contour consisting of following segments
$\Gamma_{1}^{T}=[1.5, 1.5+Ti]$, $\Gamma_{2}^{T}=[-0.5, -0.5+Ti]$,$\Gamma_{h}^{T}=[-0.5+Ti, 1.5+Ti]$,$\Gamma_{3}=[-0.5, 0.5-r]\cup \Gamma_{r} \cup [0.5+r, 1.5]$,where $\Gamma_{r}$ is a upper half circle with radius $r$ and centered at $z=0.5$.
Set $$\zeta_{a}^{T}(s)=\sum_{0<Im \rho_{m}<T}\frac{1}{(\frac{\rho_{m}-\frac{1}{2}}{i})^s}$$

We have that:
\begin{eqnarray}
\zeta_{a}^{T}(s)=\frac{1}{2\pi i}\int_{\Gamma_{r}^T}\frac{1}{(\frac{z-\frac{1}{2}}{i})^s}\frac{\xi^{'}}{\xi}(z)dz
=\frac{1}{2\pi i}(\int_{\Gamma_{1}^T}+\int_{\Gamma_{2}^T}+\int_{\Gamma_{h}^T}+\int_{\Gamma_{3}})\frac{1}{(\frac{z-\frac{1}{2}}{i})^s}\frac{\xi^{'}}{\xi}(z)dz\nonumber
\end{eqnarray}

we denote above four integrals by $\Sigma_1$,$\Sigma_2$,$\Sigma_3$,$\Sigma_4$ respectively.
let's first figure out $\Sigma_3$, by the appendix ..

$$\left|\Sigma_{3}\right|\leq C_{1}\int_{-0.5}^{1.5}\frac{1}{\left|\frac{\sigma+iT-\frac{1}{2}}{i}\right|^{Res}}\left|\frac{\xi^{'}}{\xi}(\sigma+iT)\right|d\sigma$$
$$=O(T^{-Res}lnT)$$
Thus $$lim_{T\rightarrow 0}\Sigma_{3}=0$$

Since $\xi(z)=\xi(1-z)$, $\frac{\xi^{'}}{\xi}(z)=-\frac{\xi^{'}}{\xi}(1-z)$

\begin{eqnarray}
\int_{\Gamma_{2}^T}\frac{1}{(\frac{z-\frac{1}{2}}{i})^s}\frac{\xi^{'}}{\xi}(z)dz=
\int_{J_{2}^T}\frac{1}{(\frac{\frac{1}{2}-z'}{i})^s}\frac{\xi^{'}}{\xi}(z')dz'\nonumber
\end{eqnarray}

Where $J_{2}^T$ is the straight line from $z'=15-iT$ to $z'=1.5$ and $z'=1-z$

Since $\left|\frac{\xi^{'}}{\xi}(z)\right|=O(lnt) $, we can deform the path 

$\Gamma_{1}^T$ and $J_{2}^T$ to the real axis $[1.5,+\infty]$ when we let $T$ approaches

to infinity, we get 
\begin{eqnarray}
lim_{T\rightarrow \infty}\int_{\Gamma_{2}^T}\frac{1}{(\frac{z-\frac{1}{2}}{i})^s}\frac{\xi^{'}}{\xi}(z)dz=\int_{1.5}^{+\infty}\frac{1}{(\frac{t-\frac{1}{2}}{i})^s}\frac{\xi^{'}}{\xi}(t)dt\nonumber
\end{eqnarray}

and

\begin{eqnarray}
lim_{T\rightarrow \infty}\int_{J_{2}^T}\frac{1}{(\frac{z-\frac{1}{2}}{i})^s}\frac{\xi^{'}}{\xi}(z)dz=-\int_{1.5}^{+\infty}\frac{1}{(\frac{\frac{1}{2}-t}{i})^s}\frac{\xi^{'}}{\xi}(t)dt\nonumber
=-(-1)^{-s}\int_{1.5}^{+\infty}\frac{1}{(\frac{t-\frac{1}{2}}{i})^s}\frac{\xi^{'}}{\xi}(t)dt\nonumber
\end{eqnarray}

To summarize, we have 

\begin{eqnarray}
\zeta_{a}(s)=\frac{sin\frac{\pi}{2}s}{\pi}\int_{1.5}^{+\infty}\frac{1}{(t-\frac{1}{2})^s}\frac{\xi^{'}}{\xi}(t)dt
+\frac{1}{2\pi i}\int_{\Gamma_{3}}\frac{1}{(\frac{z-\frac{1}{2}}{i})^s}\frac{\xi^{'}}{\xi}(z)dz\nonumber\\
=\frac{\sin\frac{\pi}{2}s}{\pi}\int_{0.5+r}^{+\infty}\frac{1}{(t-\frac{1}{2})^s}\frac{\xi^{'}}{\xi}(t)dt 
+\frac{1}{2\pi i}\int_{\Gamma_{r}}\frac{1}{(\frac{z-\frac{1}{2}}{i})^s}\frac{\xi^{'}}{\xi}(z)dz\label{a}
\end{eqnarray}

We can not simply  let $r$ approach to zero in \ref{a} because the singuarity at $z=0.5$ will cause divergent result.

To get through such divergence, we use integration by parts in \ref{a},we have 
\begin{eqnarray}
\zeta_{a}(s)=\frac{sin\frac{\pi}{2}s}{\pi(s-1)}\int_{0.5+r}^{+\infty}\frac{\psi(t)}{(t-\frac{1}{2})^{s-1}}dt\nonumber
+\frac{i^s}{2\pi i(s-1)}\int_{\Gamma_{r}}\frac{\psi(z)}{(\frac{z-\frac{1}{2}}{i})^{s-1}}dz
\end{eqnarray} 
Where $\psi(z)=\frac{d}{dz}(\frac{\xi^{'}}{\xi}(z))$
When $1<Res<2$, $$\left|\int_{\Gamma_{r}}\frac{\psi(z)}{(\frac{z-\frac{1}{2}}{i})^{s-1}}dz\right|=O(r^{2-Res})$$
Let $r$ approaches to zero in \ref{a}, we obtain

\begin{eqnarray}
\zeta_{a}(s)=\frac{sin\frac{\pi}{2}s}{\pi(s-1)}\int_{0.5}^{+\infty}\frac{\psi(t)}{(t-\frac{1}{2})^{s-1}}dt
=\frac{sin\frac{\pi}{2}s}{\pi(s-1)}\int_{0.5}^{+\infty}\frac{\psi(t+\frac{1}{2})}{t^{s-1}}dt\nonumber
\end{eqnarray}
When $1<Res<2$
\end{proof}

  To use above theorem to get analytic continuation, we are going to evaluate $\zeta_{a}(s)$ explicitly.
  
  First of all, we write $\psi(s)$ as sum of three functions as follows:
  
 \begin{eqnarray}
 \psi(s)=-\frac{1}{(s-1)^2}+J(s)+ \frac{d}{ds}(\frac{\xi^{'}}{\xi}(s))\nonumber
 \end{eqnarray}
  
  Where $$J(s)=-\frac{1}{s^2}+\frac{1}{2}\frac{d}{ds}(\frac{\Gamma'}{\Gamma}(s))=\sum_{n=1}^{\infty}\frac{1}{(s+2n)^2}$$
When $1<Res<2$, we have

\begin{eqnarray}
\int_{0}^{\infty}\frac{J(t+\frac{1}{2})}{t^{s-1}}dt=\sum_{n=1}^{\infty}\int_{0}^{\infty}\frac{1}{(t+2n+\frac{1}{2})^{2}t^{s-1}}dt\nonumber
\end{eqnarray}

To carry on integration by parts for every term, we assume tempoarily that $Res<1$ and get
\begin{eqnarray}
\int_{0}^{\infty}\frac{1}{(t+2n-\frac{1}{2})t^{s-1}}dt=(1-s)\int_{0}^{\infty}\frac{1}{(t+2n+\frac{1}{2})t^{s}}dt
=\frac{\pi(s-1)}{sin\pi s(2n+\frac{1}{2})^s}\label{h}
\end{eqnarray}
Thus we have
\begin{eqnarray}
\int_{0}^{\infty}\frac{J(t+\frac{1}{2})}{t^{s-1}}dt=\frac{\pi(s-1)}{sin\pi z}\sum_{n=1}^{\infty}\frac{1}{(2n+\frac{1}{2})^s}\nonumber
=\frac{\pi(s-1)}{2^{s}sin\pi s}\zeta(s,\frac{5}{4})\nonumber
\end{eqnarray}

Where $\zeta(s,q)$ is Hurwitz zeta function.
To evaluate integral $$I=\frac{1}{s-1}\int_{\Gamma_{r}'}\frac{1}{(z-\frac{1}{2})^{2}z^{s-1}}dz$$
Where the path $$\Gamma_{r}'=[0,\frac{1}{2}-r]\cup \Gamma_{r}\cup[\frac{1}{2}+r,+\infty)$$
using integration by parts, $$I=-\int_{\Gamma_{r}'}\frac{1}{(z-\frac{1}{2})z^{s}}dz$$

Let's deform $\Gamma_{r}'$ to immaginary axis $[0,i\infty)$, then
 
$$I=\frac{-1}{i^{s-1}}\int_{0}^{\infty}\frac{1}{(iy-\frac{1}{2})y^s}dy $$
Set $J=\int_{0}^{\infty}\frac{1}{(iy-\frac{1}{2})y^s}dy$, we have
\begin{eqnarray}
J=-\frac{1}{2}\int_{0}^{\infty}\frac{y^{-s}}{y^{2}+\frac{1}{4}}dy-i\int_{0}^{\infty}\frac{y^{1-s}}{y^{2}+\frac{1}{4}}dy\label{b}
\end{eqnarray}

By the substituation of  $y=\sqrt{t}$ in \ref{b} and according appendix, we obtain
\begin{eqnarray}
J=-\frac{1}{4}\int_{0}^{\infty}\frac{t^{\frac{-s-1}{2}}}{t+\frac{1}{4}}dt-\frac{i}{2}\int_{0}^{\infty}\frac{t^{\frac{-s}{2}}}{t+\frac{1}{4}}dt\nonumber
=\frac{\pi 2^{s-1}}{cos\frac{\pi}{2}s}+i\frac{\pi 2^{s-1}}{sin\frac{\pi}{2}s}\nonumber
=\frac{i\pi 2^s}{sin\pi s}e^{-\frac{i\pi}{2}s}\nonumber
\end{eqnarray}

To summarize, when $1<Res<2,$the final integral representation of $\zeta_{a}(s)$ can be described as follows:
\begin{eqnarray}
\zeta_{a}(s)=\frac{2^{s-1}}{cos(\frac{\pi}{2}s)e^{i\pi s}}-\frac{\zeta(s,\frac{5}{4})}{cos(\frac{\pi}{2}s)2^{s+1}}
+\frac{sin\frac{\pi}{2}s}{(s-1)\pi}\int_{\Gamma_{r}'}\frac{J_(z+\frac{1}{2})}{z^{s-1}}dz \label{c}
\end{eqnarray}
Where $J_{2}(s)=\frac{d}{ds}(\frac{\zeta'}{\zeta}(s))$

Because $J_{2}(t)=O(2^{-t})$, the domain of third term of $\zeta_{a}(s)$ can be extended to $-\infty<Res<2$ and

furthermore,  we have following theorem:
\begin{theorem}
Acid zeta function $\zeta_{a}(s)$ is a meromorphic function of whole complex plane and has an explicit formula as
\ref{c} when $-\infty<Res<2$
\end{theorem}

\bigskip

\section{Adjoint acid zeta function}

To relate acid zeta function to the Riemann hyothesis, we need to define another function so called adjoint acid zeta function as follows:

{\bf Definition.}Let $N(T)$ be distribution function of untrival Riemann zeros whose imaginary part are greater than zero, set $\zeta_{a}^{*}(s)=\int_{0}^{\infty}t^{-s}dN(t)$ , where $Res>1$. We call $\zeta_{a}^{*}(s)$ adjoint acid 
zeta function.
\bigskip

{\bf Remark.} Obviously, if RH is true, acid zeta function is equal to adjoint acid zeta function and vise versa.
 Comparing explicit formula of acid zeta function with adjoint acid zeta function, we will find when $1<Res<2$ acid zeta function contains a term which is increasing exponently along vertical direction and ajoint acid zeta function is bounded
along vertical direction, such difference is impressive and will make us doubt the truth of RH seriously though rigious
disproof haven't come up yet.
\bigskip
 
 In this section, we will dicuss analytic continuation as we did for acid zeta function and get a ''half explicit" formula
 of ajoint acid zeta function, first of all, we need following lemma as prepartion.
\begin{lemma} 
Let $f(z)$ be a holomorphic function in the region $Res>\frac{1}{2}-\delta$($\delta$ is some positive number) 
and $\left|f''(z)\right|=O(\frac{1}{\left|z\right|})$, $f(\bar{z})=\bar{f}(z)$. Set $M(t)=Imf(\frac{1}{2}+it)$,we have that
\begin{eqnarray}
\int_{0}^{\infty}\frac{M''(t)}{t^{s-1}}dt=-sin\frac{\pi}{2}s\int_{0}^{\infty}\frac{f''(\frac{1}{2}+t)}{t^{s-1}}dt\nonumber
\end{eqnarray}
Where $1<Res<2$
\end{lemma}
\begin{proof}
Since $f(\bar{z})=\bar{f}(z)$, $$M(t)=Imf(\frac{1}{2}+it)=\frac{1}{2i}[f(\frac{1}{2}+it)-f(\frac{1}{2}-it)]$$
and $$M''(t)=\frac{i}{2}[f''(\frac{1}{2}+it)-f''(\frac{1}{2}-it)]$$
Thus,
\begin{eqnarray}
\int_{0}^{\infty}\frac{M"(t)}{t^{s-1}}dt=\frac{i}{2}(\int_{0}^{\infty}\frac{f"(\frac{1}{2}+it)}{t^{s-1}}dt\nonumber
-\int_{0}^{\infty}\frac{f''(\frac{1}{2}-it)}{t^{s-1}}dt)\label{d}
\end{eqnarray}
We just compute the first integral of \ref{d} and the second one can be dealed with similarly.
Set $z=it$ we get 
\begin{eqnarray}
\int_{0}^{\infty}\frac{f''(\frac{1}{2}+it)}{t^{s-1}}dt=i^{s-2}\int_{0}^{i\infty}\frac{f''(\frac{1}{2}+z)}{z^{s-1}}dz\nonumber
\end{eqnarray}
Deforming path back to real axis, we get that it's equal to $$-i^{s}\int_{0}^{i\infty}\frac{f"(\frac{1}{2}+t)}{t^{s-1}}dt$$
Similarly,
\begin{eqnarray}
\int_{0}^{\infty}\frac{f''(\frac{1}{2}-it)}{t^{s-1}}dt=-i^{-s}\int_{0}^{i\infty}\frac{f''(\frac{1}{2}+t)}{t^{s-1}}dt\nonumber
\end{eqnarray}
That follows the lemma
\end{proof} 
 
Since $N(t)=\frac{1}{\pi}Img(\frac{1}{2}+it)$, where $g(s)=arg(s-1)\pi^{-\frac{s}{2}}\Gamma(\frac{s}{2}+1)\zeta(s)$
we have $$\zeta_{a}^{*}(s)=\frac{1}{\pi}\int_{1}^{\infty}t^{-s}darg(-\frac{1}{2}+it)+\int_{1}^{\infty}t^{-s}dM(t)+
\int_{1}^{\infty}t^{-s}dS_{0}(t)$$

 Here $$M(t)=\frac{1}{\pi}arg\pi^{-\frac{1}{4}+\frac{it}{2}}\Gamma(-\frac{5}{4}+\frac{it}{2})$$ and
 $$S_{0}(t)=\frac{1}{\pi}arg\zeta(\frac{1}{2}+it)$$
 Using integration by parts
 \begin{eqnarray}
 \int_{1}^{\infty}t^{-s}dM(t)=\frac{M'(1)}{s-1}+\frac{1}{s-1}\int_{1}^{\infty}t^{1-s}M''(t)dt\nonumber
\end{eqnarray}
 Since $M''(t)=O(\frac{1}{t})$, assuming $1<Res<2$,we have
 $$\int_{1}^{\infty}t^{1-s}M''(t)dt=\int_{0}^{\infty}t^{1-s}M''(t)dt-\int_{0}^{1}t^{1-s}M''(t)dt$$
let $$f(z)=\pi^{-\frac{z}{2}-1}\Gamma(\frac{z}{2}+1)$$ by the lemma 3 and \ref{h}, we get
 
\begin{eqnarray}
\frac{1}{s-1}\int_{0}^{\infty}t^{1-s}M''(t)dt=\frac{-sin\frac{\pi}{2}s}{\pi(s-1)}\int_{0}^{\infty}t^{1-s}f''(\frac{1}{2}+t)dt=-\frac{\zeta(s,\frac{5}{4})}{2^{s+1}cos\frac{\pi}{2}s}\nonumber
\end{eqnarray}
 and
\begin{eqnarray}
 \frac{1}{\pi}\int_{0}^{\infty}t^{-s}darg(-\frac{1}{2}+it)=-\frac{1}{2\pi}\int_{0}^{\infty}\frac{1}{t^{s}(\frac{1}{4}+t^2)}dt\nonumber\\
=-\frac{1}{4\pi}\int_{0}^{\infty}\frac{1}{t'^{\frac{s+1}{2}}(\frac{1}{4}+t')}dt'\nonumber
=\frac{2^{s-1}}{cos\frac{\pi}{2}s}\nonumber
\end{eqnarray}
We have used substituation $t'=t^2$ in the last two steps,thus $$\frac{1}{\pi}\int_{1}^{\infty}t^{-s}darg(-\frac{1}{2}+it)=\frac{2^{s-1}}{cos\frac{\pi}{2}s}-\frac{1}{\pi}\int_{0}^{1}t^{-s}darg(-\frac{1}{2}+it)$$
Simiarly,
\begin{eqnarray}
\int_{1}^{\infty}t^{-s}dS_{0}(t)=-S_{0}(1)+s\int_{1}^{\infty}\frac{S_{0}(t)}{t^{s+1}}dt\nonumber
=-S_{0}(1)+s\int_{1}^{\infty}t^{-s-1}dS_{1}(t)\nonumber\\
=-S_{0}(1)-sS_{1}(1)+s(s+1)\int_{1}^{\infty}\frac{S_{1}(t)}{t^{s+2}}dt\nonumber
\end{eqnarray}

Where $S_{1}(T)=\int_{0}^{T}S_{0}(t)dt+c_1$ and $c_1$ is a constant , $S_{1}(t)=O(logt)$\cite{L}.
Finally, we get "half explicit" formula for $\zeta_{a}^{*}(s)$ when $Res>-1$ as followings:
\begin{theorem}
When $Res>-1$,
\begin{eqnarray}
\zeta_{a}^{*}(s)=\frac{2^{s-1}}{cos\frac{\pi}{2}s}-\frac{\zeta(s,\frac{5}{4})}{2^{s+1}cos\frac{\pi}{2}s}+s(s+1)\int_{1}^{\infty}\frac{S_{1}(t)}{t^{s+2}}dt+R(s)\nonumber
\end{eqnarray}
Where $$R(s)=-\int_{0}^{1}t^{s}d(M(t)+arg(-\frac{1}{2}+it))-S_{0}(1)-sS_{1}(1)$$
\end{theorem}
\bigskip
\section{Relationship between acid zeta function and ajoint acid zeta function}
In this section, we will present a theorem which connects acid zeta function and ajoint acid zeta function,
as a consequence, we can make "infinite number of numerical" equivalences of Riemann hypothesis.

\begin{theorem}
When $Res>-1$, we have that
\begin{eqnarray}
\zeta_{a}(s)=\zeta_{a}^{*}(s)+\sum_{n,m=1}^{\infty}\frac{(-1)^{n}s(s+1)\cdots(s+2n-1)}{(2n)!}\frac{(\sigma_{m}-\frac{1}{2})^{2n}}{\lambda_{m}^{s+2n}}\label{e}
\end{eqnarray}
Where $\rho_{m}=\sigma_{m}+i\lambda_{m}$ run over all off critical line zeros of Riemann zeta function ,which have positive
imaginary part.
\end{theorem}
\begin{proof}
When $Res>1$
\begin{eqnarray}
\zeta_{a}(s)=\sum_{m}\frac{1}{(\frac{\rho_{m}-\frac{1}{2}}{i})^s}
=\sum_{m}\frac{1}{\lambda_{m}}(1+\frac{i(\frac{1}{2}-\sigma_{m})}{\lambda_{m}})^{-s}\nonumber\\
=\sum_{m}\frac{1}{\lambda_{m}^{s}}[1+\sum_{k=1}^{\infty}\frac{((-s)(-s-1)\cdots(-s-k+1)}{k!}(\frac{i( \frac{1}{2}-\sigma_{m})}{\lambda_{m}})^{k}\nonumber]\\
=\sum_{m}\frac{1}{\lambda_{m}^{s}}[1+\sum_{n=1}^{\infty}\frac{(-i)^{k}(-s)(-s-1)\cdots(-s-k+1)}{k!}(\frac{( \frac{1}{2}-\sigma_{m})}{\lambda_{m}})^{k}]\nonumber\\
=\sum_{m}\frac{1}{\lambda_{m}^{s}}+\sum_{n,m=1}^{\infty}\frac{(-1)^{n}s(s+1)\cdots(s+2n-1)}{(2n)!}\frac{(\sigma_{m}-\frac{1}{2})^{2n}}{\lambda_{m}^{s+2n}}\nonumber
\end{eqnarray}
 The last step is due to appearance in pair for $\rho_m$ and $1-\rho_m$.
It's easy to find that above formula hold for $Res>-1$
\end{proof}
With above formula, we can get

\begin{corollary}
RH is true if and only if 
$$\frac{d}{ds}\zeta_{a}(s)|_{s=0}=\frac{d}{ds}\zeta_{a}^{*}(s)|_{s=0}$$
\end{corollary}
\begin{proof}
Taking derivative on both sides of \ref{e}and let $s=0$,we have
\begin{eqnarray}
\frac{d}{ds}\zeta_{a}(s)|_{s=0}= \frac{d}{ds}\zeta_{a}^{*}(s)|_{s=0}-\frac{1}{2}\sum_{m=1}^{\infty}ln[1+(\frac{\sigma_m-\frac{1}{2}}{\lambda_m})^2]\nonumber
\end{eqnarray}
\end{proof}
which immediately follows our corollary.

\bigskip
{\bf Remark.} 
It's interesting to figure out $\zeta_{a}'(0)$ and the special value can be respresented in term of 
 Euler constant, our equivalence look very close to the Volchkov criteria \cite{V} and comparison will also be interesting ,besides we can prove RH is true iff there exists a positive integer $n>1$ such that
 $\zeta_{a}(n)=\zeta_{a}^{*}(n)$,in other words, we get infinite number of RH equivalences.

{\bf Appendix}

\bigskip

When $-1<Res<0$,
\begin{eqnarray}
\int_{0}^{\infty}\frac{x^s}{x+a}dx=-\frac{\pi}{sin\pi s}a^s\nonumber
\end{eqnarray}
Where $a>0$

\end{document}